\documentclass[a4paper,11pt]{amsart}          
\usepackage{amsfonts,amsmath,latexsym,amssymb} % these packages are required
\usepackage{amsthm}      
\usepackage[dvipsnames]{xcolor}

\usepackage{float}
\usepackage{tikz}
\usepackage{circuitikz}
\usepackage{hyperref}
\usepackage{mathtools}
\usepackage{enumitem}
\allowdisplaybreaks

\newtheorem{theorem}{Theorem}[section]      
\newtheorem{lemma}[theorem]{Lemma}               
\newtheorem{corollary}[theorem]{Corollary}

\theoremstyle{definition}

\newtheorem{definition}[theorem]{Definition}
\newtheorem{assumption}[theorem]{Assumption}

\newtheorem{remark}[theorem]{Remark}

\DeclareMathOperator\supp{supp}
\DeclareMathOperator\dist{dist}
\DeclareMathOperator{\sgn}{sgn}
\renewcommand{\a}{a}%admittance
\renewcommand{\Re}{\operatorname{Re}}%real part
\renewcommand{\Im}{\operatorname{Im}}%imaginary part

\newcommand{\cp}{\operatorname{cap}}
\newcommand{\abs}[1]{\left|#1\right|}%imaginary part

\renewcommand{\L}{\mathcal{L}}
\newcommand{\D}{\mathcal{D}}
\newcommand{\Q}{\mathcal{Q}}
\newcommand{\cfreq}{s}

\author[A. Muranova]{Anna Muranova}
\address{Anna Muranova: University of Warmia and Mazury in Olsztyn,
Faculty of Mathematics and Computer Science,
ul. Słoneczna 54, 
10-710 Olsztyn, Poland} 

 \email{anna.muranova@uwm.edu.pl}

\title[$m$-sectorial Laplacians and recurrence]{$m$-sectorial discrete Laplacians and recurrence of complex-weighted graphs}

\begin{document}

\maketitle

\begin{abstract}
We consider complex-weighted graphs, whose edge weights belong to a sector in the complex plane. We show that the corresponding Dirichlet Laplacian is $m$-sectorial, and, hence, generates a contractive holomorphic $C_0$-semigroup. Further, it is shown that every sectorial complex-weighted graph can be extended to an electrical network, where by electrical networks we mean graphs, whose edge weights are holomorphic functions, arising from physical admittances. This result allows us to    establish  convergence results for the infinite complex-weighted graphs, e.g. convergence of solutions of Dirichlet problems  and convergence of complex-valued capacities on a finite exhaustion. Finally, we define a recurrence for complex-weighted graphs, and, using the convergence results,  give its characterizations in terms of functional spaces, capacity, Green's function, resolvents of the Dirichlet Laplacian and properties of the Neumann Laplacian.
\end{abstract}

{\footnotesize
\smallskip

}

%\tableofcontents

\section{Introduction}
The discrete Laplacian on graphs is a classical topic. Firstly it was considered on combinatorial graphs, and than extended to weighted graphs with positive weights. For further details we refer reader to monographs \cite{Chung, Grigoryan, LPW, KellerLenzWojcBook}. The Laplacian is known to be closely related to random walks, Markov chains, flows and electrical networks, see \cite{DoyleSnell, LPW, Woess09, Woess00}. 

It is known in physics, that an electrical network, connected to an AC power source, can be modeled using a graph, whose edge weights are rational functions of special structure. More precisely,
the electrical network, can be represented as a graph, whose each edge $\{x,y\}$  is endowed with a complex-valued rational function
\begin{equation*}
\a^{(\cfreq)}(x,y)=\dfrac{\cfreq}{L_{xy}\cfreq^2 +R_{xy}\cfreq+D_{xy} },
\end{equation*}
 where $R_{xy}\ge 0$ is the
resistance of this edge, $L_{xy}\ge 0$ is the inductance, $D_{xy}\ge 0$ is the inverse
capacitance and at least one of $R_{xy},L _{xy}, D_{xy}$ is not equal $0$. Usually, one take the domain of this function to be the complex right half-plane $\cfreq\in \Bbb C$ with $\Re \cfreq>0$,
and call the function {\em admittance} (see \cite{Brune, ChenTeplyaev, Feynman, Muranova3, Muranova1, MuranovaWoess}). This model gives rise to a family of graphs and, consequently, Laplace operators  $\L^{(\cfreq)}$ on the set of functions $f: X\to \Bbb C$, where $X$ is a finite or countable set of vertices  of a graph, see \cite{Muranova3, Muranova1}. Moreover, in \cite{MuranovaWoess} it is proven, that recurrence and transience is well-defined for each family, i.e. it does not depend on $\cfreq$.

In this paper we consider complex-weighted graphs, whose edge weight $b: X\times X\to \Bbb C$ maps to a sector, i.e.
$$
|\Im b(x,y)|\le c\cdot \Re b(x,y),
$$ 
for some $c\ge 0$ and any edge $\{x,y\}$. 
We prove that the corresponding Dirichlet Laplacian is an $m$-sectorial operator, and, hence, by Lumer-Phillips theorem, generates  contractive holomorphic $C_0$-semigroup. 

The main result of this paper is Theorem \ref{thm::main} which states, that for each  complex-weighted graph, whose edge weights belong to a sector, there exist an electrical network and $\cfreq_0\in \Bbb C$ with $\Re \cfreq_0>0$ such that  $b(x,y)=\a^{(\cfreq_0)}(x,y)$ for all edges $\{x,y\}$, i.e. any complex-weighted graph can be ``extended''~to an electrical network. This result can be considered as a generalization of the classical correspondence between real-weighted graphs and electrical networks with resistors, see, e.g. monograph \cite{DoyleSnell} for more details. 

Further, we introduce a recurrence for the complex-weighted graphs. Theorem \ref{thm::main} allows us to introduce a capacity of infinite complex-weighted graph,  and characterize the recurrence of complex-weighted graphs in terms  the capacity, Theorem \ref{thm::reccap}. Then we introduce a Green's function on the complex-weighted graphs and, using holomorphicity of capacity and Theorem \ref{thm::main}, we show, that the Green's function is well-defined and also characterizes recurrence. Finally, we show that the Neumann Laplacian on the complex-weighted graphs is also $m$-sectorial and relate its properties to recurrence. In fact, Theorem \ref{thm::main}  allows us to use uniform convergence of holomorphic functions for complex-weighted graphs instead of monotone convergence, used for real-weighted graphs and not applicable for complex values. 

The paper is organized as follows. Section \ref{sect::preliminaries} contains preliminaries and basic results on sectorial complex-weighted graphs. Section \ref{sect::finiteapprox} elaborates on finite approximations of the infinite complex-weighted graphs. Section \ref{sect::elnetwork} explains a relation between the complex-weighted graphs and the electrical networks. Section \ref{section:recurrence} is dedicated to the definition of recurrence for complex-weighted graphs and its characterizations in terms of capacity, Green's function, resolvents of the Dirichlet Laplacian and properties of the Neumann Laplacian.

\section{Preliminaries}
\label{sect::preliminaries}
\subsection{Complex-weighted graphs and formal Laplacian}

\begin{definition}\label{def::graph}
A {\em graph (with a measure)} is a triple $(X,E,m)$, where $X$ is an at most countable set, 
$$
E
 \subset \{\{x,y\}\in X\times X\;\mid\;x\ne y\}
 $$ 
 is a given subset of unordered pairs of elements of $X$, and $m:X\to\Bbb R^+$. The elements of the set $X$ are called {\em vertices}, the elements of $E$ are called {\em edges}, and $m$ is the {\em measure}.
\end{definition}

For any $x,y\in X$ we write $x\sim y$ if $\{x,y\}\in E$. A graph is called {\em locally finite} if the set
$
\{y\;\mid\;y\sim x\}
$ is finite for all $x\in X$.  A {\em path} between  vertices $x,y\in X$ is a finite sequence $(x_0,\ldots,x_{n}), n\in \mathbb N\cup\{0\}$,  of vertices such that
$$
x=x_0\sim x_1\sim x_2\sim \dots \sim x_n=y.
$$

A set $ K\subseteq X $ is called {\em connected}, if there exists a path between any two vertices of it, and so we call  the graph \emph{connected} if $ X $ is connected.

\begin{assumption}\label{a:graph}  In this paper we assume that all graphs $ (X,E,m) $ are locally finite and connected.
\end{assumption}

Let us denote
$$
\Bbb H_r:=\{\cfreq \in \Bbb C\;\mid\;\Re \cfreq>0\}.
$$

\begin{definition}\label{def::complexweightedgraph}
A {\em complex-weighted graph} is a graph $(X,E,m)$, whose each edge is equipped with the weight $b:E\to \Bbb H_r$, satisfying the following property: there exists $c\ge 0$ such that
\begin{equation}\label{bsect}
\left|\Im b(x,y)\right|\le c \cdot \Re b(x,y).
\end{equation}
We call this property {\em sectoriality} and $c$ is a {\em sectoriality constant}.
\end{definition}
We assume $b(x,y)\equiv 0$ if there is no edge between $x$ and $y$. Hence, any complex-weighted graph is uniquely determined by the triple $(X,b,m)$, and we will refer to this triple as the complex-weighted graph.

\begin{assumption}
In this paper the complex weightsof a graph always satisfy a sectoriality \eqref{bsect}, i.e. all the weights of the graph belong to a sector of the complex plane.
\end{assumption}

The sectoriality  leads to many useful features of the graph and its energy form, the most basic of which are discussed in this section below. Moreover, the corresponding Dirichlet Laplacian is $m$-sectorial and generates a contractive holomorphic $C_0$-semigroup, see Section \ref{section::DirichletLaplacian}.

Let $X$ be an at most countable set and $(X,b,m)$ be a complex-weighted graph. We denote by $C(X)$ the set of all complex-valued functions on $X$.

The {\em formal Laplacian} is acting on  any function $f\in C(X)$ via
$$
\L_{b,m} f(x):=\dfrac{1}{m(x)}\sum_{y\in X} (f(x)-f(y)) b(x,y)
$$
We will omit subscripts $b$ and(or) $m$ when they are  clear from the context.

\subsection{Energy form and formal sesquilinear form }

 Let $\ell^2(X,m)$ be the set of functions
$$\{f\in C(X)\;\mid\;\sum_{x\in X}|f(x)|^2m(x)<\infty\big\},
$$
with a scalar product
$$
(f\mid g)=\sum_{x\in X} f(x)\overline{g(x)}m(x).
$$
Obviously, $\ell^2(X,m)$  is a Hilbert space.

Further, let us define a subspace $\mathcal D(X)\subset C(X)$ by

\begin{align*}
\mathcal D(X)&=\Big\{f\in C(X)\;\mid\;\sum_{x,y\in X}\abs{f(x)-f(y)}^2 |b(x,y)|<\infty\Big\}\\
&=\Big\{f\in C(X)\;\mid\; \sum_{x,y\in X}\abs{f(x)-f(y)}^2 b(x,y)\substack{ \mbox{ converges for some (all) }\\\mbox{ order of summations}}\Big\}\\
&=\Big\{f\in C(X)\;\mid\;\sum_{x,y \in X}\abs{f(x)-f(y)}^2 \Re b(x,y) <\infty\Big\},
\end{align*}
where the equalities follow from the sectoriality \eqref{bsect}.

The {\em (formal) energy} of a function is defined for any $f\in \mathcal D(X)$ as 
\begin{align*}
\mathcal Q_b(f)&=\dfrac12 \sum_{x,y\in X}\abs{f(x)-f(y)}^2b(x,y).
\end{align*}
We will omit the subscript $b$, if it is clear from the context.

%Note that $f\in\D(X)$ if and only if $\overline f\in \D(X)$.

Note that 
\begin{equation*}%\label{eq::Reqf}
\Re \mathcal Q(f)=\dfrac12 \sum_{x,y\in X}\abs{f(x)-f(y)}^2\Re b(x,y)\ge 0,
\end{equation*}
and
$$
\Im \mathcal Q(f)=\dfrac12 \sum_{x,y\in X}\abs{f(x)-f(y)}^2\Im b(x,y).
$$

Further,  the following lemma follows immediately from the sectoriality~\eqref{bsect}.
\begin{lemma}\label{lem::formalQmapstosector}
Let $(X,b,m)$ be a complex-weighted graph. The  form $\mathcal Q=\mathcal Q_b$ with the domain $\D(X)$ is sectorial, i.e. it satisfies
$$
|\Im \mathcal Q(f)|\le  c \cdot\Re \mathcal Q(f),
$$
for some $c\ge0 $ for any $f\in \D(X)$.
\end{lemma}
The next several lemmas are almost immediate corollaries of the definition of the form $\Q$ and are presented here for further references.

\begin{lemma}\label{lem::complexf}
Let $(X,b,m)$ be a complex-weighted graph. For any $f\in C(X)$ the following are equivalent:
\begin{itemize}
\item
$f\in \D(X)$,
 \item
 the both  functions $\Im f, \Re f\in \D(X)$.

 \end{itemize} 
 Moreover, if  any of the equivalent conditions is satisfied, the following holds:
\begin{equation}\label{eq::QfQRefQImf}
\mathcal Q(f)=\mathcal Q(\Re f)+\mathcal Q(\Im f).
\end{equation}
\end{lemma}
\begin{proof}
The statement follows from the sectoriality \eqref{bsect}, the fact that
$$
|f(x)-f(y)|^2=|\Re f(x)-\Re f(y)|^2+|\Im f(x)-\Im f(y)|^2.
$$
and the definition of $\mathcal Q$.
\end{proof}

The corresponding (formal) sesquilinear form is defined for any $f,g\in \D(X)$ by
\begin{align*}
\mathcal Q(f,g)&=\dfrac12 \sum_{x,y\in X}{(f(x)-f(y))}\overline{(g(x)-g(y))}b(x,y)
\end{align*}
The following inequality, which is a consequence of the to Cauchy-Schwarz inequality, is known for sectorial forms:

\begin{equation}\label{eq::CSforsect}
|\mathcal Q(f,g)|\le (1+c) \left(\Re \mathcal Q(f)\right)^\frac{1}{2} \left(\Re \mathcal Q(g)\right)^\frac{1}{2}.
\end{equation}

\bigskip
 Let us denote by $C_c(X)$ the subset of $C(X)$ of all complex-valued functions on $X$ with finite support.

\begin{lemma}
Let $(X,b,m)$ be a complex-weighted graph. Then $C_c(X)$ is dense in $\ell^2(X,m)$.
\end{lemma}

\begin{proof}
Let $f\in \ell^2(X,m)$. Let $(K_n)_{n\in \Bbb N}$ be a finite exhaustion of $(X, b, m)$, i.e. $K_n\subset X$ is finite and connected,  $K_{n}\subset K_{n+1}$ for any $n\in \Bbb N$ and $\cup_{n\in \Bbb N} K_n=X$. Let $f_n:=f\mid_{K_n}$, $n\in \Bbb N$. Then
$$
\|f-f_n\|^2=\sum_{x\in X\setminus K_n}|f(x)|^2m(x)\to 0,
$$
as $n\to\infty$.
\end{proof}

\begin{theorem}[Green's formula]
Let $(X,b,m)$ be a complex-weighted graph. Then 
\begin{enumerate}[leftmargin=*]
\item[{\em (a)}]
For all $f\in~\D(X)$ and $\phi \in C_c(X)$ the following holds
\begin{align*}
&\sum_{x\in X}\L f(x) \phi(x) m(x)=\sum_{x\in X}\L \phi(x) f(x)m(x)\\
&=\dfrac{1}{2}\sum_{x,y\in X}b(x,y)(\phi(x)-\phi(y))(f(x)-f(y)).
\end{align*}

\item[{\em (b)}]
For all $f\in \D\cap \ell^2(X,m)$ and $\phi \in C_c(X)$ the following holds
\begin{equation*}
\mathcal Q(\phi, f)=(\L \phi\mid f).
\end{equation*}
If, in addition, $\L f\in  \ell^2(X,m)$, then
\begin{equation*}
\mathcal Q(f,\phi)=( \L f\mid \phi).
\end{equation*}
\end{enumerate}
\end{theorem}

\begin{proof}
Due to the locally finitness of graph all the sums are finite. The equalities follow by direct computations, see e.g \cite[Lemma 7]{Muranova1}. 
\end{proof}

\subsection{Dirichlet Laplacian and corresponding form on infinite graphs}\label{section::DirichletLaplacian}

 In this section we introduce the Dirichlet Laplacian on a  complex-weighted graph. The approach is similar to the one for  the real-weighted Dirichlet Laplacian, see \cite{KellerLenzWojcBook}.  
 
 Let $(X,b,m)$ be an infinite graph and $\mathcal Q$ be its energy form. We denote by $\D_0$  the subspace of functions $f\in C(X)$ for which
there exists a sequence $(\phi_n)_{n\in \Bbb N}\subset C_c(X)$ with $\phi_n\to f$ pointwise and $\Re \mathcal Q(f-\phi_n)\to 0$ as $n\to \infty$. We say that {\em the  sequence $(\phi_n)_{n\in \Bbb N}\subset C_c(X)$ approximates $f$ in $\D_0$}.

\begin{remark}\label{rem::d0ReQ}
Note that $\Re\mathcal Q(f-\phi_n)\to 0$ as $n\to \infty$ is equivalent to $\mathcal Q(f-\phi_n)\to 0$ as $n\to \infty$ due to the sectoriality of $\Q$,  Lemma \ref{lem::formalQmapstosector}.
\end{remark}

Obviously, $C_c(X)\subset \D_0$. Moreover,

\begin{lemma}
For any complex-weighted graph $(X,b,m)$ the following inclusion of functional spaces holds: $\D_0 \subset \D$.
\end{lemma}

\begin{proof}
Let $f\in\D_0$. Then there exist $(\phi_n)_{n\in \Bbb N}\subset C_c(X)$ with $\phi_n\to f$ pointwise and $\Re \mathcal Q(f-\phi_n)\to 0$ as $n\to \infty$. Hence, for any $\varepsilon>0$ there exist $N\in \Bbb N$ such that $\Re \mathcal Q(f-\phi_N)<\varepsilon$, i.e
$$
\dfrac12\sum_{x,y\in X\setminus B_1(\supp \phi_N)}|f(x)-f(y)|^2\Re b(x,y)\le \Re \mathcal Q(f-\phi_N)<\varepsilon,
$$
where
$$
 B_1(\supp \phi_N)=\{x\in X\;\mid\; \mbox{there exists }y\in \supp \phi_N \mbox{ such that }x\sim y\}.
$$
Hence, $\Re \mathcal Q(f)<\infty$ and $f\in\D$.
\end{proof}

The norm, associated to the form $\mathcal Q$, is defined for any $f\in \mathcal D$ as
$$
\|f\|_\mathcal Q=(\Re \mathcal Q(f)+\|f\|^2)^{\frac{1}{2}}.
$$

\begin{theorem}\label{thm::dol2cQ}
For any graph  $(X,b,m)$ the following holds
$$
\mathcal D_0\cap \ell^2(X,m)=\overline{C_c(X)}^{\|\cdot\|_\mathcal Q}.
$$
\end{theorem}

Before we prove the theorem, we prove the following lemma:
\begin{lemma}\label{lem::fRefImfH}
Let $f\in C(X)$ and $H \in \Big\{\mathcal D_0, \ell^2(X,m), \overline{C_c(X)}^{\|\cdot\|_\mathcal Q}\Big\}$. Then
$f\in H$ if and only if $\Re f, \Im f\in H$.
\end{lemma}
\begin{proof}
\begin{itemize}[leftmargin=*]
\item
Let $H=\mathcal D_0$. 
Let $(\phi_n)_n\in C_c(X)$ approximates $f$ in $\D_0$. Firstly note, that
$\phi_n\to f$ pointwise implies $\Re \phi_n\to \Re f$ and $\Im \phi_n\to \Im f$. 
Secondly, by Lemma \ref{lem::complexf}\eqref{eq::QfQRefQImf} we have
$$
\Re \mathcal Q(f-\phi_n)=\Re \mathcal Q(\Re f-\Re \phi_n)+\Re \mathcal Q(\Im f-\Im \phi_n),
$$
and, since all summands are positive and $\Q$ is sectorial, we get
 $\Re f, \Im f \in~D_0$. Moreover, they can be approximated in $\mathcal D_0$ by real-valued functions. The other direction follows from the fact that $\D_0$ is a vector space.
\item
Let $H=\ell^2(X,m)$. Then we obtain
$$
\|f\|^2=\sum_{x\in X}|f(x)|^2m(x)=\sum_{x\in X}(|\Re f(x)|^2+|\Im f(x)|^2)m(x)
$$
from where the statement follows.
\item
Let $H=\overline{C_c(X)}^{\|\cdot\|_\mathcal Q}$.
By Lemma \ref{lem::complexf}\eqref{eq::QfQRefQImf}  and definition of $\ell^2$-norm we get
\begin{align*}
&\Re \mathcal Q(f-\phi_n)+\|f-\phi_n\|^2=\Re \mathcal Q(\Re f-\Re \phi_n)+\Re \mathcal Q(\Im  f-\Im \phi_n)\\
&+\|\Re f-\Re \phi_n\|^2+\|\Im f-\Im \phi_n\|^2,
\end{align*}
from where the statement follows due to the definition of $\|\cdot\|_\mathcal Q$ and positivity of all summands. 
\end{itemize}
\end{proof}

\begin{proof}[Proof of Theorem \ref{thm::dol2cQ}]
Let $f\in \overline{C_c(X)}^{\|\cdot\|_\mathcal Q}$. Then there exists a sequence  $(\phi_n)_{n\in \Bbb N}\subset C_c(X)$ such that
$$
\|f-\phi_n\|_\mathcal Q=(\Re \mathcal Q(f-\phi_n)+\|f-\phi_n\|^2)^{\frac{1}{2}}\to 0 \mbox{ as }n\to\infty.
$$
Therefore, $\phi_n\to f$ in $\ell^2(X,m)$ and, hence, $f\in \ell^2(X,m)$ and $\phi_n\to f$ pointwise. Moreover, $\Re \mathcal Q(f-\phi_n)\to 0$ and $f\in \mathcal D_0$. From all the above follows $\overline{C_c(X)}^{\|\cdot\|_\mathcal Q}\subset \mathcal D_0\cap \ell^2(X,m)$.

Let $f\in D_0\cap \ell^2(X,m)$. Then, by Lemma \ref{lem::fRefImfH} we obtain that $\Re f,\Im f \in D_0\cap \ell^2(X,m)$ and can be approximated in $\D_0$ by real-valued functions. Hence, since $\Re \Q_b(\cdot)=\Q_{\Re b}(\cdot)$ one can apply the reasoning for positive forms on real-weighted graph $(X,\Re b,m)$, see  \cite[proof of Theorem 1.19]{KellerLenzWojcBook}, to show that $\Re f,\Im f \in \overline{C_c(X)}^{\|\cdot\|_\mathcal Q}$ which implies $f \in \overline{C_c(X)}^{\|\cdot\|_\mathcal Q}$ by Lemma~\ref{lem::fRefImfH}. Therefore, $ \mathcal D_0\cap \ell^2(X,m)\subset\overline{C_c(X)}^{\|\cdot\|_\mathcal Q}$ and the proof is finished.

\end{proof}
We define the form $Q^{(D)}$ as a restriction of the form $\Q$ to $\mathcal D_0\cap \ell^2(X,m)$. We call $Q^{(D)}$ the {\em form corresponding to the Dirichlet Laplacian}.
\begin{corollary}
The form $Q^{(D)}$ with the domain $D(Q^{(D)})=\mathcal D_0\cap \ell^2(X,m)$ is a densely defined closed sectorial form in $\ell^2(X,m)$.
\end{corollary}
\begin{proof}
The sectoriality follows from Lemma \ref{lem::formalQmapstosector}. The density is clear, since $C_c(X)\subset D(Q^{(D)})$ is dense in $\ell^2(X,m)$. Finally, the form is closed due to Theorem~\ref{thm::dol2cQ}.
\end{proof}

By the general theory of sectorial forms and operators \cite{Kato}, the form $Q^{(D)}$ defines a unique $m$-sectorial operator $L^{(D)}$ with a domain $D(L^{(D)})\subset D(Q^{(D)})$, see \cite[p. 322, Theorem 2.1]{Kato}.  We call the operator $L^{(D)}$ the {\em Dirichlet Laplacian}.

\begin{lemma}\label{lem::LK=mathcalL}
Let $(X,b,m)$ be a complex-weighted graph. Let $L^{(D)}$ be its Dirichlet Laplacian. Then
$$
L^{(D)}f(x)=\mathcal Lf(x),
$$
for all $f\in D(L^{(D)})$, for any $x\in X$.

\end{lemma}
\begin{proof}
 By the definition of an operator, associated with a form, for any $f\in D(L^{(D)}),$ and $\phi\in C(X)\subset D(Q^{(D)})$ we have 
$$
(L^{(D)}f\;\mid\; \phi)=Q^{(D)}(f,\phi)=\mathcal Q(f,\phi)=\sum_{x\in X}\L f(x)\overline {\phi(x)}m(x),
$$
where the last equality is due to the Green's formula. Now taking $\phi:=1_x/m$ the statement follows.
\end{proof}

By Lumer-Phillips theorem for generation of holomorphic semigroups (see e. g. \cite[p. 492, Theorem 1.24]{Kato}) we obtain that $-L^{(D)}$ generates a contractive holomorphic $C_0$-semigroup $T(t)$ of angle $\theta=\pi/2-\arg c$, where $c$ is the sectoriality constant \eqref{bsect}. Then by integral representation of the resolvent (see, e.g. \cite[p. 55 Theorem 1.10]{EngelNagel}) the following holds
\begin{equation*}\label{eq::resolvent}
(L^{(D)}+\alpha)^{-1}=\int_0^\infty e^{-\alpha t}T(t)dt,
\end{equation*}
for all $\alpha\in \Bbb H_r$. We will elaborate more on this equality in the next sections.

\section{Finite approximations of complex-weighted graphs}
\label{sect::finiteapprox}
Let $(X,b,m)$ be an infinite complex-weighted graph and let $K\subset X$ be a finite connected subset, fixed for this section. Let us denote 
$$
C_c(K)=\{f\in C_c(X)\;\mid\; \supp f\subset K\},
$$
and 
$$
C(K)=\{f: K\to \Bbb C\}.
$$
We denote by $\ell^2(K,m_K)$ a Hilbert space
of functions $\{f\in C(K)\}$ with the scalar product 
$$
(f\mid g)_K=\sum_{x\in K}f(x)\overline{g(x)}m(x).
$$

Let us define a sesquilinear form $\widehat Q_K$ on $\ell^2(K,m_k)$ by
$$
\widehat Q_K(f,g)=\mathcal Q(i_Kf,i_Kg),
$$
where   $i_K:C(K)\to C_c(K)$ is an extension of $f\in C(K)$ to $X$ by setting $i_Kf$ be identically zero outside $K$. It follows from sectoriality of $\Q$, Lemma \ref{lem::formalQmapstosector}, that $\widehat Q_K$ is a sectorial form. We denote a sectorial operator, associated to $\widehat  Q_K$, by $\widehat L_K$  and call it the  {\em (Dirichlet) Laplacian on $K$}. We denote by $D(\widehat L_K)$ the domain of $\widehat L_K$. The operator $\widehat L_K$ is $m$-sectorial and, due too the finitness of $K$,  $D(\widehat L_K)=\ell^2(K,m_K)$  (see, e.g. \cite[p. 322, Theorem 2.1]{Kato}). 

Following the same outline as in the proof of Lemma \ref{lem::LK=mathcalL} we can show, that 
$$
\widehat L_Kf(x)=(\mathcal L\circ i_K)(f)(x),
$$  
for all $f\in \ell^2(K,m_K)$ and any $x\in K$.

Now we are in the position to show that the operators $\widehat L_K+\alpha$ are invertible for all $\alpha\ge 0$ (Theorem \ref{LKinvert}). We start with the following lemma.
\begin{lemma}\label{lem::dirprr}
Let $(X,b,m)$ be a complex-weighted graph, $K\subset X$ is finite connected, $x\in K$. Let $\alpha\ge 0$. Then the problem
\begin{equation}\label{prob::dirpralpha}
\begin{cases}
v_\alpha(x)=r,\\
(\L +\alpha)v_\alpha=0\mbox{ on  }K\setminus\{x\},\\
v_\alpha= 0 \mbox{ on  }X\setminus K,
\end{cases}
\end{equation}
has a unique solution $v_\alpha\in C(X)$ for any $r\in \Bbb R$. Moreover, 
\begin{equation}\label{eq::aLplusalphaQ}
r\cdot (\L v_\alpha+\alpha v_\alpha)(x)m(x)=Q(v_\alpha)+\alpha \|v_\alpha\|^2
\end{equation}
 and the following estimate holds
\begin{equation}\label{est::Qvr}
|\mathcal Q(v_\alpha)|+\alpha \|v_\alpha\|^2\le (1+c^2)|r|^2\left(\sum_{\substack{y\in X:\\y\sim x}}|b(x,y)|+\alpha\cdot m(x)\right),
\end{equation}
where $c$ is the sectoriality constant \eqref{bsect}.
\end{lemma}
The estimate \eqref{est::Qvr} is an important point, since we need it to relate Green's function and recurrence in Section \ref{section:recurrence}. Note, that this estimate does not depend on $K$.

\begin{proof}
Due to the Green's formula and \eqref{prob::dirpralpha}, if a solution $v_\alpha$ exists, we have 
$$
r\cdot (\L v_\alpha+\alpha v_\alpha)(x)m(x)= ((\L +\alpha)v_\alpha\;\mid\;v_\alpha)=\mathcal  Q(v_\alpha,v_\alpha)+\alpha \|v_\alpha\|^2,
$$
which proves \eqref{eq::aLplusalphaQ}.
Since $\mathcal  Q(v_\alpha,v_\alpha)=0$ if and only if $v_\alpha\equiv 0$, the latest is the only solution of \eqref{prob::dirpralpha}  for $r=0$. Hence, the corresponding matrix is invertible and \eqref{prob::dirpralpha} has a unique solution $v_\alpha$ for any $r\in \Bbb R$. 

To prove the estimate, let us note that by Green's formula
$$
 (\L v_\alpha+\alpha v_\alpha)(x)m(x)=((\L +\alpha)v_\alpha\;\mid\;1_x)=\mathcal  Q(v_\alpha,1_x)+\alpha r\cdot m(x),
$$
Combining the latest with \eqref{eq::aLplusalphaQ} we obtain:
\begin{equation}\label{eq::QvrQvr1}
|\mathcal Q(v_\alpha)|+\alpha \|v_\alpha\|^2\le |r|\cdot|\mathcal  Q(v_\alpha,1_x)|+\alpha|r|^2\cdot m(x).
\end{equation}

Now, using the inequality
$$2|z_1 z_2 | \le\varepsilon |z_1|^2 +\dfrac{1}{\varepsilon} |z_2 |^2 \mbox{ for any  }z_1 , z_2 \in  \Bbb C, \varepsilon>0,$$
we get
\begin{multline}\label{eq::Qvr1x}
|\mathcal  Q(v_\alpha,1_x)|\le \dfrac12\sum_{y,z\in X}|v_\alpha(y)-v_\alpha(z)||1_x(y)-1_x(z)||b(y,z)|\\\le \dfrac{\varepsilon}{4}\sum_{y,z\in X}|v_\alpha(y)-v_\alpha(z)|^2|b(y,z)|+\dfrac{1}{4\varepsilon}\sum_{y,z\in X}|1_x(y)-1_x(z)|^2|b(y,z)|.
\end{multline}
and, similarly, 
\begin{equation}\label{eq::alphaeps}
\alpha |r|\cdot m(x)\le   \dfrac{\varepsilon}{2}\alpha|r|^2 \cdot m(x)+ \dfrac{\alpha}{2\varepsilon}\cdot m(x)\le   \dfrac{\varepsilon}{2}\alpha \|v_\alpha\|^2+ \dfrac{\alpha}{2\varepsilon}\cdot m(x)
\end{equation}
From the other hand, note that for any $y,z\in X$ we have due to the sectoriality~\eqref{bsect}
$$
\Re b(y,z)\ge (1+c^2)^{-\frac12}|b(y,z)|,
$$
 and, hence, we obtain
\begin{multline*}
|\mathcal  Q(v_\alpha)| \ge \Re \mathcal  Q(v_\alpha)=\dfrac12\sum_{y,z\in X}|v_\alpha(y)-v_\alpha(z)|^2 \Re b(y,z)\\\ge \dfrac{1}{2(1+c^2)^{\frac12}}\sum_{y,z\in X}|v_\alpha(y)-v_\alpha(z)|^2 |b(y,z)|. 
\end{multline*}
Combining this with \eqref{eq::QvrQvr1}, \eqref{eq::Qvr1x}, \eqref{eq::alphaeps} and using the fact that $1/(1+c^2)^\frac12\le 1$ we get
\begin{align}\label{eq::2(1+c2)}
&\dfrac{1}{2(1+c^2)^{\frac12}}\sum_{y,z\in X}|v_\alpha(y)-v_\alpha(z)|^2 |b(y,z)|+\dfrac{1}{(1+c^2)^{\frac12}}\alpha \|v_\alpha\|^2\nonumber\\
&\le|\mathcal  Q(v_\alpha)| +\alpha \|v_\alpha\|^2\le
 |r|\cdot|\mathcal  Q(v_\alpha,1_x)|+\alpha|r|^2\cdot m(x) 
\\& \le\dfrac{\varepsilon |r|}{4}\sum_{y,z\in X}|v_\alpha(y)-v_\alpha(z)|^2|b(y,z)|+   \dfrac{\varepsilon |r|}{2}\alpha \|v_\alpha\|^2+\dfrac{|r|}{2\varepsilon}\widetilde C \nonumber,
\end{align}
where
$$\widetilde C:=\dfrac12\sum_{y,z\in X}|1_x(y)-1_x(z)|^2|b(y,z)|+\alpha\cdot m(x).
$$
 Now plugging in \eqref{eq::2(1+c2)} $\varepsilon=(1+c^2)^{-\frac12}|r|^{-1}$ 
we obtain from the first and the third line after simplification the following:
\begin{equation*}
\dfrac{1}{4}\sum_{y,z\in X}|v_\alpha(y)-v_\alpha(z)|^2 |b(y,z)|+\dfrac{1}{2}\alpha \|v_\alpha\|^2\\\le\dfrac{(1+c^2)|r|^2}{2}\widetilde C,
\end{equation*}
from where the estimate \eqref{est::Qvr} follows, since 
$$
\widetilde C=\frac12\sum_{y,z\in X}|1_x(y)-1_x(z)|^2|b(y,z)|+\alpha\cdot m(x)=\sum_{\substack {y\in X:}\\{y\sim x}}|b(x,y)|+\alpha\cdot m(x).
$$
\end{proof}

\begin{theorem}\label{LKinvert}
Let $(X,b,m)$ be a complex-weighted graph, $K\subset X$ is finite connected, $x\in K$. Let $\alpha\ge 0$. Then the problem
\begin{equation}\label{prob::LKplusalpha}
(\widehat L_K+\alpha)\widehat v_\alpha=1_x
\end{equation}
has a unique solution $v\in \ell^2(K,m_K)$.  Moreover, the operator $\widehat L_K+\alpha$ is invertible on $\ell^2(K,m_K)$.
\end{theorem}
\begin{proof}
Let $v_\alpha$ be as in Lemma \ref{lem::dirprr} with $r=1$.
It follows immediately from  \eqref{eq::aLplusalphaQ} that $(\L  v_\alpha+ \alpha  v_\alpha)(x)\ne 0$,  since $v_\alpha\not\equiv 0$.
Hence, the function $\widehat v_\alpha\in \ell^2(K,m_K)$ defined by
$$
\widehat v_\alpha(y)=\dfrac{v_\alpha(y)}{(\L   v_\alpha+\alpha  v_\alpha)(x)},
$$
for all $y\in K$ is the solution of \eqref{prob::LKplusalpha}. To show the uniqueness firstly note, that $\widehat v_\alpha(x)\ne 0$ for any solution $\widehat v_\alpha$ of \eqref{prob::LKplusalpha}. Indeed, if $\widehat v_\alpha(x)=0$, then, by the definition of the Dirichlet Laplacian through the form $\widehat  Q_K$, we have
$$
0=((\widehat  L_K+\alpha)\widehat v_\alpha\;\mid\;\widehat v_\alpha)=\widehat  Q_K(\widehat v_\alpha)+\alpha \|\widehat v_\alpha\|^2,
$$
i.e. $\widehat v_\alpha\equiv 0$ which is a contradiction. Then uniqueness follows from the fact, that for any $\widehat v_\alpha$, solution of  \eqref{prob::LKplusalpha}, the function 
$$
 \widetilde v_\alpha(y)=\dfrac{\widehat v_\alpha(y)}{\widehat v_\alpha(x)},
$$
is the solution of \eqref{prob::dirpralpha} for $r=1$, which is unique by  Lemma \ref{lem::dirprr}.

Finally note that since the operator $(L_K+\alpha)$ is linear, it follows from  the existence of the unique solution of \eqref{prob::LKplusalpha} that it is  invertible.
\end{proof}

With slight abuse in notation, we can consider  $\ell^2(K,m_K)$ being a subspace of $\ell^2(X,m)$, by identifying functions in $C(K)$ with functions in $C_c(K)$.
We denote by $L_K$ the operator defined on $\ell^2(X,m)$ by
\begin{equation}\label{eq::LK}
L_Kf:=i_K(\widehat L_K)(f\mid_K)
\end{equation}
for any $f\in \ell^2(X,m)$. We call this operator the {\em (Dirichlet) Laplacian with respect to $K$}.
The same concerns the notations for resolvent, i.e we denote
\begin{equation}\label{eq::LKres}
(L_K+\alpha)^{-1}f:=i_K(\widehat L_K+\alpha)^{-1}(f\mid_K),
\end{equation}
for any $\alpha\in \Bbb H_r$, for any   $f\in \ell^2(X,m)$.

To introduce the main result of this section we remind that for a sectorial form $Q$ in $H$ with the domain $D(Q)$ a subset $D\subset Q$ is called a {\em core}, if $D$ is dense in $D(Q)$ (with respect to the norm $\|\cdot\|_Q=(\Re Q(\cdot)+\|\cdot\|^2)^\frac12$), see, e.g. \cite[Remark 3.4]{Arendt} or \cite[p. 317, Theorem 1.21]{Kato}.

Moreover, for the convenience of the reader we cite the following theorem by Vogt and Voigt.

\begin{theorem}[Theorem 1.1 in \cite{VogtVoigt}]\label{thm::VV}
Let $H$ be a complex Hilbert space, and let $Q$ be a sectorial form in $H$. For $n\in \Bbb N$ let $Q_n$ be a form with $D(Q_n)\subset D(Q)$. Assume that there exists $c\in\Bbb R$ such that 
\begin{equation*}
|\Im (Q_n(u)-Q(u))|\le c\cdot \Re (Q_n(u)-Q(u)),
\end{equation*}
for all $u\in D(Q_n)$ for all $n\in \Bbb N$.
Let $D$ be a core for $Q$, and suppose that for all $u\in D$ there exists a sequence $(u_n)_{n\in\Bbb N}\subset D(Q), u_n\in D(Q_n)$ for all $n\in \Bbb N$, such that $u_n\to u$ in $D(Q)$ and $Q_n(u_n)-Q(u)\to 0$ as $n\to \infty$.

Let $A$ be the linear relation associated with $Q$, and let $A_n$ be the linear relation associated with $Q_n$, for $n\in \Bbb N$. Then $A_n$ converges to $A$ in the strong resolvent sense, i.e. $(A_n+\alpha)^{-1}\to (A+\alpha)^{-1}$ (as $n\to\infty$) strongly for all $\alpha>0$.

\end{theorem}
By {\em linear relation, associated  with a form $Q$}, it is meant in \cite{VogtVoigt} the 
corresponding $m$-sectorial operator, if the form is densely defined. Otherwise, the form defines an $m$-sectorial operator $\widetilde A$ on the subspace $\widetilde H$ of the Hilbert space $H$. Obviously, this operator  $\widetilde A$ can be identified with a subset of $H\times H$, which we also denote by $\widetilde A$, i.e. we can write with a slight abuse of notations:
$$
\widetilde A=\{(x,y)\in H\times H\;\mid\;y = \widetilde Ax\}.
$$
Then  the associated linear relation on $H$ is defined by $A:=\widetilde A \bigoplus (\{0\} \times \widetilde H^\perp )$, where the direct sum is an orthogonal direct sum in $H\times H$. Moreover, the linear relation $(A+\alpha)^{-1}$ is defined by
$$
(A+\alpha)^{-1}:=\{(y+\alpha x,x)\in H\times H\;\mid\;(x,y)\in A\},
$$
for any $\alpha>0$.
Then the  linear relations $(A_n+\alpha)^{-1}$ and  $(A+\alpha)^{-1}$ in the Theorem \ref{thm::VV} are operators, see  \cite[Remark 3.2(d)]{VogtVoigt}.

\bigskip
A sequence $(K_n)_{n\in \Bbb N}$ is a {\em finite exhaustion} of a complex-weighted graph $(X,b,m)$ if $K_n\subset X$ are finite connected, $K_n\subset K_{n+1}$ for all $n\in \Bbb N$ and $X=\cup_{n\in \Bbb N}K_n$. Then the following theorem shows the relation between the Dirichlet Laplacians $L_{K_n}$ and the Dirichlet Laplacian $L^{(D)}$.

\begin{theorem}\label{thm::LDeqlimLn}
Let $(X,b,m)$ be a complex-valued graph. Let $(K_n)_{n\in \Bbb N}$  be a finite exhaustion of $(X,b,m)$ and $L_n$ be Dirichlet Laplacians with respect to $K_n, n\in \Bbb N$. Then
\begin{equation*}
(L^{(D)}+\alpha)^{-1}=\lim_{n\to\infty}({L_n}+\alpha)^{-1},
\end{equation*}
in the strong resolvent sense for all $\alpha>0$.
\end{theorem}

\begin{proof}
Since $C_c(X)$ is dense in $\ell^2(X,m)$, the Dirichlet Laplacian $L^{(D)}$  is an operator, associated to $Q^{(D)}$, and, hence, a linear relation. Further, it is clear,
that the operator $(L_n+\alpha)^{-1}$, $\alpha>0$, is the needed  linear relation by its definition \eqref{eq::LKres}, since $\widehat L_n$ is the operator, associated with $\widehat Q_n$. Note that $\widehat Q_n(u)=Q^{(D)}(u)$ for all $u\in D(\widehat Q_{K_n})$. Now the statement is an immediate corollary of Theorem~\ref{thm::VV}, since $C_c(X)$ is a core of $Q^{(D)}$ by Theorem \ref{thm::dol2cQ}, and, hence the sequence $u_n$ can be taken such that $u_n=u$ eventually. 
\end{proof}
Note, that the operator $L_K$, defined by \eqref{eq::LK} is not the linear relation, corresponding to $\widehat L_K$, since $L_K$ is not surjective on $\ell^2(X,m)$.

\section{Electrical networks}
\label{sect::elnetwork}
\subsection{Definitions and main properties}
It is known, see, e.g. \cite{Brune, Feynman, Muranova1} that an electrical network, connected to an AC source, can be represented as a graph, whose  edges are equipped with complex weights, depending on a parameter. The following formal definition was introduced in \cite{Muranova3} and \cite{Muranova1}. 
\begin{definition}\label{def::elnetwork}
An {\em electrical network} is a graph $(X,E,m)$, whose each edge is equipped with an {\em admittance},
i.e. a function
\begin{equation*}
\a^{(\cfreq)}(x,y)=\dfrac{\cfreq}{L_{xy}\cfreq^2 +R_{xy}\cfreq+D_{xy} },
\end{equation*}
where $L_{xy},R_{xy},D_{xy}\ge 0$ and at least one of them is not zero. 
\end{definition}

We assume $\a^{(\cfreq)}(x,y)\equiv 0$ if there is no edge between $x$ and $y$. Then $\a^{(\cfreq)}:X\times X\to \Bbb R$ for any $\cfreq \in \Bbb H_r$  and  any electrical network is uniquely determined by the triple $(X,\a^{(\cfreq)},m)$ to which we  will refer as an electrical network.

\begin{theorem}\label{thm:ImsRes}
The electrical network  $(X,\a^{(\cfreq)},m)$ defines a complex-weighted graph with the secoriality constant
$$
c=\dfrac{|\Im \cfreq|}{\Re \cfreq},
$$
for any $\cfreq\in\Bbb H_r$.
\end{theorem}
\begin{proof}
By \cite[Lemma 4]{Muranova3} we have
$$
\dfrac{\Re \a^{(\cfreq)}(x,y)}{|\a^{(\cfreq)}(x,y)|}\ge  \dfrac{\Re \cfreq}{|\cfreq|},
$$
which can be rewritten as
$$
{(\Re^2 \a^{(\cfreq)}(x,y))}{(\Re^2\cfreq+\Im^2 \cfreq)}\ge {(\Re^2 \cfreq)}{(\Re^2\a^{(\cfreq)}(x,y)+\Im^2 \a^{(\cfreq)}(x,y))},
$$
which is equivalent to
$$
{(\Re^2 \a^{(\cfreq)}(x,y))}{(\Im^2 \cfreq)}\ge {(\Re^2 \cfreq)}{(\Im^2 \a^{(\cfreq)}(x,y))},
$$
from where immediately follows that the weight $ \a^{(\cfreq)}(x,y)$ satisfies sectoriality  \eqref{bsect} with
$$
c = \dfrac{|\Im \cfreq|}{\Re \cfreq}.
$$
\end{proof}

Therefore, we the  formal Laplacian $\L^{(\cfreq)}$ and sesquilinear form $\mathcal Q^{(\cfreq)}$ can be introduced for electrical networks. We will omit superscript $(\cfreq)$ when it is clear from the context.

\subsection{Complex-weighted graph ``extension'' to an electrical network}
The real-weighted graphs are known to be related to electrical networks with resistors, see, e.g. \cite{DoyleSnell}. In this section we show, that any complex-weighted graph (satisfying sectoriality property \eqref{bsect}) can be ``extended'' to an electrical network. Being an interesting result on its own, this also allows us to prove convergence results in Section \ref{section:recurrence}. The following theorem can be considered as the main result of this paper. 
\begin{theorem}\label{thm::main}
Let $(X,b,m)$ be a complex-weighted graph, $c$ be its sectoriality constant \eqref{bsect}.
Then for any $\cfreq_0\in \Bbb H_r$ with $\dfrac{|\Im \cfreq_0|}{\Re \cfreq_0}\ge c$, there exist  an electrical network  $(X,a^{(\cfreq)},m)$ such that 
$$
a^{(\cfreq_0)}(x,y)=b(x,y)
$$
for any $x,y\in X$.
\end{theorem}

\begin{proof}
 We need to prove that the equation 
\begin{equation*}
\dfrac{\cfreq_0}{L_{xy}\cfreq_0^2+R_{xy}\cfreq_0+D_{xy}}=b(x,y).
\end{equation*}
 has a solution in terms of $L_{xy},R_{xy}, D_{xy}$, where all the variables are positive, for any $b(x,y)\ne 0$, satisfying sectoriality \eqref{bsect}. Then the corresponding electrical network will be defined by
 $$
\a^{(\cfreq)}(x,y):= \dfrac{\cfreq}{L_{xy}\cfreq^2+R_{xy}\cfreq+D_{xy}},
 $$
 for all $x,y\in X$ with $b(x,y)\ne 0$ and $\a^{(\cfreq)}(x,y):=0$ otherwise.
 
 Note, that if $\Im b(x,y)=0$, then $L_{xy}=D_{xy}=0$ and $R_{xy}=1/b(x,y)$ is the needed solution. So assume $\Im b(x,y)\ne 0$. Let us denote $b:=b(x,y)$, omit subscripts ${xy}$ for the proof and solve:
$$
\dfrac{\cfreq_0}{L\cfreq_0^2+R\cfreq_0+D}=b.
$$
Assume, the denominator is not zero (which will hold if the solution will be non-negative and positive for at least one of the variables $L,R,D$). Without loss of generality we can assume that $\Re \cfreq_0=1$ (otherwise scale $L,R,D$). Hence, let us write $\cfreq_0=1+\omega i$, where $|\omega|\ge c$. Then we should solve
\begin{equation*}
b(L(1+i\omega)^2+R(1+i\omega)+D)={1+i\omega},
\end{equation*}
which we rewrite as
\begin{equation}\label{eq::beq1}
L(1+i\omega)^2+R(1+i\omega)+D=\dfrac{{(1+i\omega)}\overline b}{|b|^2}.
\end{equation}
Let $b_1=\Re b$ and $b_2=\Im b$. By assumption $0<|b_2|\le c\cdot b_1$. Then the equation \eqref{eq::beq1}, being split into real and imaginary part, is equivalent to:
$$
\begin{cases}
(1-\omega^2)L+R+D=\dfrac{b_1+\omega b_2}{|b|^2},\\
2\omega L+\omega R=\dfrac{b_1\omega - b_2}{|b|^2}.
\end{cases}
$$
Since there are $2$ equations and $3$ variables, let us assume that $D=D_0$ is fixed (we will choose it later) and find a solution for $(L_0, R_0)$, using the Cramer's rule. Let
$$
A=\begin{pmatrix}
1-\omega^2&1\\
2\omega&\omega
\end{pmatrix},
A_L=
\begin{pmatrix}
\dfrac{b_1+\omega b_2}{|b|^2}-D_0&1\\
\dfrac{b_1\omega - b_2}{|b|^2}&\omega
\end{pmatrix},
$$
and
$$
A_R=
\begin{pmatrix}
1-\omega^2&\dfrac{b_1+\omega b_2}{|b|^2}-D_0\\
2\omega&\dfrac{b_1\omega - b_2}{|b|^2}
\end{pmatrix}.
$$
Hence we get
$\det A=-\omega^3-\omega$ and $\sgn(\det A)=-\sgn \omega$. Further,
\begin{equation}\label{eq::detAL}
\det A_L=\dfrac{(1+\omega^2)b_2}{|b|^2}-\omega D_0,
\end{equation}
$$
\det A_R=\dfrac{-(1+\omega^2)(b_1\omega+b_2)}{|b|^2}+2\omega D_0.
$$
Since $L=\dfrac{\det A_L}{\det A}$ and $R=\dfrac{\det A_R}{\det A}$ should be non-negative, we should require $\sgn(\det A_L), \sgn(\det A_R)\in \{0, -\sgn \omega\}$, i.e. we should choose $D_0$ such that
\begin{equation}\label{eq::D0}
\dfrac{(1+\omega^2)b_2}{\omega |b|^2}\le D_0\le \dfrac{(1+\omega^2)(b_1\omega + b_2)}{2 \omega |b|^2}. 
\end{equation}
Firstly note, that $|b_1\omega|=|b_1||\omega|\ge b_1c\ge|b_2|$, while $\sgn b_1\omega=\sgn \omega$ (since $b_1\ge c|b_2|>0$), and, hence, $\sgn(b_1\omega + b_2)=\sgn \omega$ or $\sgn(b_1\omega+b_2)=0$, i.e. the right hand side of the inequality \eqref{eq::D0} is non-negative. If $\sgn b_2=-\sgn \omega$, then the left hand side is negative and there exists a needed $D_0\ge 0$ (note that if $D_0=0$, then $L_0\ne 0$ since $\det A_L\ne 0$ due to \eqref{eq::detAL} and the assumption $b_2\ne 0$). So assume $\sgn b_2=\sgn \omega$. In this case since $b_1>0$, we obtain 
$$
\dfrac{b_1\omega+b_2}{\omega}=\dfrac{|b_1\omega+b_2|}{|\omega|}=\dfrac{|b_1||\omega|+|b_2|}{|\omega|}\ge \dfrac{|b_1|\cdot c+|b_2|}{|\omega|}\ge \dfrac{2|b_2|}{|\omega|}= \dfrac{2b_2}{\omega},
$$
from where immediately follows that right-hand side of  \eqref{eq::D0} is greater than or equal  to the left-hand side and, hence, there exists $D_0>0$, satisfying~\eqref{eq::D0}, since $b_2\ne 0$. 

\end{proof}

\section{Characterization of recurrence}
\label{section:recurrence}
\subsection{Concept of recurrence}

A concept of recurrence has a very important meaning in probability theory, optimal transport and theory of electrical networks.   The meaning of recurrence is the existence of a non-zero flow on a graph. For more details we refer reader to \cite[Chapter 6]{KellerLenzWojcBook} for analytic point of view and to \cite{Woess09, Woess00} for the probabilistic interpretation. Moreover, there is a branch of equivalent definitions of recurrence for real-weighted graphs, see \cite[Theorem 6.1]{KellerLenzWojcBook}, which gives the most complete description of equivalent definitions.
 We also point out that in \cite{MuranovaWoess} the concept of recurrence for electrical networks (in the sense of Definition \ref{def::elnetwork}) is considered with its applications to classical reversible random walks. 
 
 In this section we introduce recurrence for complex-weighted graphs and present some equivalent characterizations for it, following the real case, see \cite[Theorem 6.1]{KellerLenzWojcBook}. In particular, we characterize recurrence in terms of capacity, Green's function, resolvents of the Dirichlet Laplacian, and properties of the Neumann Laplacian. Although the results are the same as for real graphs, we should point out that methods of proofs are very different and heavily based on the fact that every complex-weighted graph is an electrical network, i.e. Theorem \ref{thm::main}, while in the case of real-weighted graphs proofs are usually based on monotone convergence theorem.

 We firstly remind the reader how the concept of recurrence looks for a real-weighted graph, i.e. a complex-weighted graph $(X,\mathfrak b,m)$ with $\mathfrak b: V\to~\Bbb R^+_0$. 
\begin{definition}\label{def::classicalRec}
The real-weighted graph $(X, \mathfrak b, m)$  is called {\em recurrent}, if
\begin{equation}\label{eq::infQreal}
\inf \;\{\;\Re \Q_\mathfrak b(\phi)\;|\;\phi(x)=1, \phi\in C_c(X)\;\}=0,
\end{equation}
for some (all) $x\in X$, where $\Q_\mathfrak b$ is the corresponding formal energy.
 Otherwise, the graph is called {\em transient}. 
 \end{definition}
 Note that classically one considers  functions $\phi\in C_c(X)$, taking values in $\Bbb R$, in the definition of recurrence for real-weighted graphs (see, e.g. \cite[Theorem 6.1(vii), Definition 6.2]{KellerLenzWojcBook}), but  the inverse triangle inequality implies
$$
\big|\phi(x)-\phi(y)\big|\ge \big||\phi(x)|-|\phi(y)|\big|,
$$
for any $x,y\in X$ and any $\phi\in C_c(X)$, and, hence, the classical definition is equivalent to \eqref{eq::infQreal}.

We start with the following theorem, which immediately states an equivalence of several definitions of recurrence for complex-weighted graph, see Definition \ref{def::rec} below. 
\begin{theorem}\label{thm::rec}
Let $(X,b,m)$ be a complex-weighted graph and $\Q:=\Q_b$ be a formal energy on  it. Then the following are equivalent:
\begin{itemize}
\item[{\em (i)}]
$\inf \;\{\;|\Q(\phi)|\;|\;\phi(x)=1, \phi\in C_c(X)\;\}=0,$ for some (all) $x\in X$,
\item[{\em (ii)}]
$\inf \;\{\;\Re \Q(\phi)\;|\;\phi(x)=1, \phi\in C_c(X)\;\}=0,$ for some (all) $x\in X$,
\item[{\em (iii)}]
$\inf \;\{\;\Q_{\Re b}(\phi)\;|\;\phi(x)=1, \phi\in C_c(X)\}=0,$ for some (all) $x\in X$.
\end{itemize}
\end{theorem}

\begin{proof}
We will prove the equivalences for a fixed $x\in X$. Then the equivalence of ``some'' and ``all''  will follow from (iii). Indeed, (iii) for ``some'' $x\in X$ defines a recurrence on the real-weighted graph $(X,\Re b,m)$, and  for real-weighted graphs the equivalence of ``some'' and ``all''  is known, see, e.g. \cite[Theorem 6.1 (vii)]{KellerLenzWojcBook}.

(i)$\Rightarrow$(ii) is clear.

(ii)$\Rightarrow$(i) follows from the sectoriality, Lemma~\ref{lem::formalQmapstosector}.

(ii)$\Leftrightarrow$(iii) since 
\begin{align*}
\Re \Q(\phi)&=\Re \left(\frac 12 \sum_{x,y\in X}|\phi(x)-\phi(y)|^2 b(x,y)\right)\\
&=\frac 12 \sum_{x,y\in X}|\phi(x)-\phi(y)|^2 \Re b(x,y)=\mathcal Q_{\Re b}(\phi),
\end{align*}
for any $\phi\in C_c(X)$.

\end{proof}

\begin{definition}\label{def::rec} A complex-weighted graph $(X,b,m)$ is called {\em recurrent} if any of the equivalent conditions of Theorem \ref{thm::rec} is satisfied. Otherwise, the graph is called {\em transient}.
\end{definition}

Note, that for the real-weighted graph Definition \ref{def::rec} coincides with the classical definition, described in the beginning of this section. Moreover, Theorem \ref{thm::rec} immediately implies

\begin{corollary}\label{cor::Rebrec}
A graph $(X,b,m)$ is recurrent if and only if the graph $(X,\Re b,m)$ is recurrent.
\end{corollary}

The following theorem gives one more basic characterization of recurrence.
\begin{theorem}\label{thm::1inD0}
A graph $(X,b,m)$ is recurrent if and only if $1\in \D_0.$
\end{theorem}

\begin{proof}
``$\Rightarrow$'' Let the graph $(X,b,m)$ be recurrent. Then the graph $(X,\Re b,m)$ is recurrent by Corollary \ref{cor::Rebrec}. Hence, by \cite[Theorem 6.1(i)]{KellerLenzWojcBook}, there exists a sequence $(\phi_n)_{n\in\Bbb N}\subset C_c(X)$ with $\phi_n\to f$ pointwise and 
$$
\Q_{\Re b}(1-\phi_n)=\dfrac12\sum_{x,y\in X}|\phi_n(x)-\phi_n(y)|^2\Re b(x,y)\to 0\mbox{ as }n\to \infty.
$$ 
Hence, $\Re \Q(1-\phi_n)=\Q_{\Re b}(1-\phi_n)\to 0$ and $1\in \D_0$.

``$\Leftarrow$''
Let $1\in \D_0$. Then there exists $(\phi_n)_{n\in\Bbb N}\subset C_c(X)$ such that $Q(1-\phi_n)\to 0$ as $n\to \infty$. Hence, 
$$
\Q_{\Re b}(1-|\phi_n|)\le Q_{\Re b}(1-\phi_n)=\Re Q(1-\phi_n)\to 0\mbox{ as }n\to \infty,
$$ 
and the graph $(X,\Re b,m)$ is recurrent by \cite[Theorem 6.1(i)]{KellerLenzWojcBook}, which by Corollary \ref{cor::Rebrec} implies that $(X,b,m)$ is recurrent.
\end{proof}

\subsection{Capacity}

To define the capacity of an infinite graph we start with finite approximations  and follow the approach from \cite{Grimmett} and \cite{Muranova3} for networks. 

Let $(X,b,m)$ be an infinite complex-weighted graph, $K\subset  X$ be finite and connected, $ x\in K $. Then the {\em effective capacity of $K$ at $x$} is defined as
$$
\cp_{K}(x)=\L u(x)m(x),
$$
where $u$ is the solution of the {\em Dirichlet problem:}
$$
\begin{cases}
 u(x)=1,\\
\L u(x)=0\mbox{ on  }K\setminus\{x\},\\
u= 0 \mbox{ on  }X\setminus K,
\end{cases}
$$
which is unique due to Lemma \ref{lem::dirprr}.
Note that although $ m $ appears in the definition of effective capacity of $K$ at $x$,  indeed it does not depend on $m$ since it cancels with the $ m $ appearing in the definition of the Laplacian. Moreover, by Green's formula we obtain:
$$
\cp_{K}(x)=Q(u).
$$

We remind that  a sequence of finite connected subsets $(K_n)_{\Bbb N\cup\{0\}}$ is a finite exhaustion of $(X,b,m)$ if $K_n\subset  X$, $K_n\subset K_{n+1}$ for all $n\in \Bbb N\cup\{0\}$ and $X=\cup_{\Bbb N\cup\{0\}} K_n$.

\begin{definition}
Let $(X,b,m)$ be an infinite  complex-weighted graph, $x\in X$ be a fixed vertex, $(K_n)$ be a finite exhaustion of the graph such that $x\in K_0$. Then the limit
$$
 \cp(x) :=\lim_{n\to\infty }\cp_{K_n}(x)
$$
is called an {\em (effective) capacity} of $x$.
\end{definition}

\begin{theorem}\label{thm::effCap}
The effective capacity is well-defined.
 i.e. the limit exists and does not depend on the exhaustion.
\end{theorem}

\begin{proof}
In \cite[Theorem 6(a)]{Muranova3} it is proven, that the limit of effective capacities exists in the case of electrical networks and is a holomorphic function on $\cfreq\in \Bbb H_r$  (see Definition \ref{def::elnetwork}),  for the ball exhaustion $(B_n)_{n\in \Bbb N\cup\{0\}}$, $B_n(x)=\{y\in X\mid\dist (x,y)< n\}$, where by $\dist(x,y)$ one means the lenght of the shortest path between vertices $x,y\in X$. Using the same outline of proof, one can show that the limit exists for any exhaustion $(K_n)_{n\in \Bbb N\cup\{0\}}$. Moreover, since the capacity in this case is a holomorphic function on the parameter $\cfreq\in \Bbb H_r$ and the limits coincide for all $\cfreq\in \Bbb R^+$ due to the classical theory of real-weighted graph, we immediately get, that the effective capacity of electrical network does not depend on the exhaustion, i.e. the capacity is well-defined for the electrical networks.

Finally, due to Theorem \ref{thm::main}, the capacity is well-defined for any complex-weighted graph, since the uniform limit (see again \cite[Theorem 6(a)]{Muranova3}) on any compact subset of $\Bbb H_r$ implies pointwise limit.
\end{proof}
In the case of real-weighted graphs, the solution of the Dirichlet problem on any finite $K$ minimizes the energy among all the functions, supported on $K$. This fact almost immediately implies, that recurrence  is equivalent to a zero capacity for some (all) vertices of  an infinite graph. In the complex-weighted case the minimization property for the solution of the Dirichlet problem fails. In the next theorem we show,  that the recurrence is still equivalent to the zero capacity.

\begin{theorem}\label{thm::reccap}
A graph $(X,b,m)$ is recurrent if and only if $\cp(x)= 0$ for some (all) $x\in X$.
\end{theorem}

\begin{proof}
Due to Definition \ref{def::rec} of recurrence it is enough to prove that $\cp(x)=~0$ is equivalent to
\begin{equation}\label{eq::reccapx}
\inf \;\{\;|\Q(\phi)|\;|\;\phi(x)=1, \phi\in C_c(X)\;\}=0,
\end{equation}
for a fixed $x\in X$. Then the statement for all $x\in X$ will follow from Theorem \ref{thm::rec}.

 Let $(K_n)_{n\in \Bbb N\cup \{0\}}$ be a finite exhaustion of $(X,b,m)$, $x\in K_0$.
Let $(u_n)_{n\in \Bbb N}$ be the solutions of the Dirichlet problems
\begin{equation*}
\begin{cases}
u_n(x)=1,\\
\L_b u_n(y)=0,\; y\ne x, y\in K_n,\\
u_n(y)=0,\; y\in X\setminus K_n.
\end{cases}
\end{equation*}
\begin{itemize}[leftmargin=*]
\item
Let $\cp(x)= 0$. Then $\Q(u_n)=\cp_{K_n}(x)\to 0, n\to \infty$, i.e. \eqref{eq::reccapx} holds.
\item
Let now \eqref {eq::reccapx} holds. Let us consider the Dirichlet problems on $(X,\Re b,m)$, i.e. 
\begin{equation*}
\begin{cases}
\widetilde u_n(x)=1,\\
\L_{\Re b} \widetilde u_n(a)=0,\; y\ne x, y\in K_n,\\
\widetilde u_n(y)=0,\; y\in X\setminus K_n.
\end{cases}
\end{equation*}
Note, that $ \Q_{\Re b}(\widetilde u_n)\to 0$, as $n\to \infty$, due to the classical theory of real-weighted graphs (see \cite[Theorem 6.1(xi.b)]{KellerLenzWojcBook}), since $(X, \Re b, m)$ is recurrent by  Corollary \ref{cor::Rebrec}. Hence 
$$\Re  \Q_b(\widetilde u_n) =  \Q_{\Re b}(\widetilde u_n)\to 0\mbox{ as }n\to\infty.
$$
 Further, since all functions are finitely supported, we get by Green's formula 
\begin{align*}
 &\Q_b(u_n, \widetilde u_n)=(\L_b u_n\;|\; \widetilde u_n )=\L_b u_n(x)m(x)\\
 &=\L_b u_n(x)u_n(x)m(x)=(\L_b u_n\;|\; u_n )=\Q_b(u_n).
\end{align*}
Using \eqref{eq::CSforsect}, we get
$$
\Re  \Q_b(u_n)\le | \Q_b(u_n)|= | \Q_b(u_n, \widetilde u_n)| \le(1+c)(\Re  \Q_b(u_n))^\frac 12 (\Re  \Q_b(\widetilde u_n))^\frac 12,
$$
where $c$ is the sectoriality constant \eqref{bsect} for $(X,b,m)$.
Therefore, 
$$
\Re  \Q_b(u_n)\le(1+c)^2 \Re  \Q(\widetilde u_n)=(1+c)^2  \Q_{\Re b}(\widetilde u_n)\to 0\mbox{ as }n\to\infty.
$$
and $| \Q_b(u_n)|\le (1+c)\Re  \Q_b(u_n) \to 0$ as $ n\to \infty$, where the last inequality is due to the sectoriality of $\Q$.
\end{itemize}
\end{proof}
\begin{remark}
Since $\widetilde u_n$ minimizes the energy on all the real-valued  functions, supported on $K_n$ for the graph $(X,\Re b,m)$ and $ \Re \Q_b(|u_n|)\le  \Re \Q_b(u_n)$ by convexity of absolute value,  we get
$$
\Re  \Q_b(\widetilde u_n)\le \Re  \Q_b(|u_n|)\le\Re  \Q_b(u_n)\le(1+c)^2 \Re  \Q_b(\widetilde u_n).
$$
\end{remark}

\subsection{Green's function} Now we define Green's function for complex-weighted graph. The definition requires some substantial work and is based on Theorem \ref{thm::main} and estimate \eqref{est::Qvr}. Moreover, further in this section we show, that the relation of the Green's function to recurrence is the same as in the case of real-weighted graphs, i.e.  the Green's function is equal to infinity if and only if the graph is recurrent, see Corollary \ref{cor::Greenrec}.
\begin{definition}
Let $(X,b,m)$ be a complex-weighted graph, $x,y\in X$. We define the {\em Green's function} $G:X\times X\to\Bbb C\cup \{\infty\}$ by
$$
G(x,y)=\lim_{\substack{\alpha\to 0,\\\alpha>0}}\int_0^\infty e^{-\alpha t}T(t)1_x(y)dt,
$$
where $T(t)$ is the contractive holomorphic  $C_0$-semigroup, generated by the Dirichlet Laplacian $L^{(D)}$.
\end{definition}

\begin{theorem}
Let $(X,b,m)$ be a complex-weighted graph, $x,y\in X$.  The Green's function is well-defined, i.e. the corresponding limit exists in $\Bbb C\cup\{\infty\}$.

\end{theorem}

\begin{proof}
By resolvent equality and Theorem \ref{thm::LDeqlimLn} we have:
\begin{equation}\label{eq::GxyeqL}
G(x,y)=\lim_{\substack{\alpha\to 0,\\\alpha>0}}(L^{(D)}+\alpha)^{-1}1_x(y)=\lim_{\substack{\alpha\to 0,\\\alpha>0}}\lim_{n\to\infty}({L_n}+\alpha)^{-1}1_x(y),
\end{equation}
where $L_n:=L_{K_n}, n\in \Bbb N\cup\{0\}$ are the Dirichlet Laplacians with respect to $K_n$ (see \eqref{eq::LKres}),  for some finite exhaustion $(K_n)_{ n\in \Bbb N\cup\{0\}}$ of $(X,b,m)$. Without loss of generality we can assume $x\in K_0$.  It is clear by \eqref{eq::LKres}, that
$$
(L_n+\alpha)^{-1}1_x(y)=
\begin{cases}(\widehat L_n+\alpha)^{-1}1_x(y),\quad y\in K_n\\
0, \mbox{otherwise},
\end{cases}
$$
where $\widehat L_n$ is the Laplacian on $\ell^2(K_n,m_{K_n})$. Hence, the quantity $({L_n}+\alpha)^{-1}$ is well-defined for any $x, y\in X$.  Moreover, for a fixed $y\in X$ there exist $N\in \Bbb N$ such that $y\in  K_N$ and, hence, 
\begin{equation}\label{eq::LnN}
(L_n+\alpha)^{-1}1_x(y)=(\widehat L_n+\alpha)^{-1}1_x(y) \mbox{ for all }n>N.
\end{equation}
 Therefore we need to prove the existence of the limit
$$
\lim_{\substack{\alpha\to 0,\\\alpha>0}}\lim_{n\to\infty}({L_n}+\alpha)^{-1}1_x(y)=\lim_{\substack{\alpha\to 0,\\\alpha>0}}\lim_{n\to\infty}({\widehat L_n}+\alpha)^{-1}1_x(y).
$$
By the proof of Theorem \ref{LKinvert} we have 
\begin{equation}\label{eq::LhateqL}
(\widehat {L}_n+\alpha)^{-1}1_x(y) =\dfrac{v_{\alpha,n}(y)}{(\L   v_{\alpha,n}+\alpha  v_{\alpha,n})(x)} \mbox{ for all }\alpha>0, n\in \Bbb N,
\end{equation}
where $v_{\alpha,n},n\in \Bbb N$, are the solutions of the following Dirichlet problems:
\begin{equation}\label{eq::dirpralpha1}
\begin{cases}
 v_{\alpha,n}(x)=1,\\
(\L +\alpha) v_{\alpha,n}=0\mbox{ on  }K_n\setminus\{x\},\\
 v_{\alpha,n}= 0 \mbox{ on  }X\setminus K_n.
\end{cases}
\end{equation}

By Theorem \ref{thm::main} there exist an electrical network $(X,\a^{(\cfreq)}, m)$ such that  for $s_0=1+ic$ (i.e $\Im s_0/\Re s_0=c$), where $c$ is the sectoriality constant of the graph $(X,b,m)$, the following holds
$$
\a^{(\cfreq_0)}(x,y)=b(x,y),
$$
for all $x,y\in X$. Let us extend the Dirichlet problem \eqref{eq::dirpralpha1} to the electrical network $(X,\a^{(\cfreq)},m)$, i.e. we have 
\begin{equation*}
\begin{cases}
 v_{\alpha,n}^{(\cfreq)}(x)=1,\\
(\L^{(\cfreq)} +\alpha) v_{\alpha,n}^{(\cfreq)}=0\mbox{ on  }K_n\setminus\{x\},\\
 v_{\alpha,n}^{(\cfreq)}= 0 \mbox{ on  }X\setminus K_n.
\end{cases}
\end{equation*}
By Lemma \ref{lem::dirprr} we obtain
\begin{align}\label{eq::Lv1cfreq}
\Big|\big(\L   v_{\alpha,n}^{(\cfreq)}+\alpha  v_{\alpha,n}^{(\cfreq)}\big)(x)\Big|&\le \Big|\mathcal Q\big(v_{\alpha,n}^{(\cfreq)}\big)\Big|+\alpha \| v_{\alpha,n}\|^2\notag\\
&\le \left(1+\dfrac{|\Im\cfreq|^2}{(\Re {\cfreq})^2}\right)\left(\sum_{\substack{y\in X:\\y\sim x}}\big|\a^{(\cfreq)}(x,y)\big|+\alpha\cdot m(x)\right),
\end{align}
where we have used  Theorem \ref{thm:ImsRes}.
By simple calculations, see \cite[proof of Corollary 2]{Muranova3}, we have:
\begin{equation}\label{eq::sumas}
\sum_{\substack{y\in X:\\y\sim x}} \a^{(\cfreq)}(x,y)
\le
\frac{1 + |\cfreq|^2}{\Re \cfreq}
\sum_{\substack{y\in X:\\y\sim x}} 
\left(
R _{xy} + L _{xy} +D_{xy} 
\right)^{-1}, \mbox { for any }x,y\in X.
\end{equation}
Combining this with \eqref{eq::Lv1cfreq} we get
\begin{equation}\label{eq::unifombound}
\Big|\big(\L^{(\cfreq)}   v_{\alpha,n}^{(\cfreq)}+\alpha  v_{\alpha,n}^{(\cfreq)}\big)(x)\Big|\le \dfrac{|s|^2}{(\Re s)^2} \left(\frac{1 + |\cfreq|^2}{\Re \cfreq}C_1+C_2\right),
\end{equation}
where $C_1,C_2$ do not depend on $s$ and $n$. Hence, the function $\big(\L^{(\cfreq)}   v_{\alpha,n}^{(\cfreq)}+\alpha  v_{\alpha,n}^{(\cfreq)}\big)(x)$ is uniformly bounded for any $n\in \Bbb N$  in any domain
$$
\{\Re \cfreq \ge\varepsilon, |\cfreq|\le C\}, \quad \varepsilon, C\in \Bbb R^+.
$$
Hence, by Montel's theorem (see, e.g. \cite[p. 153]{Conway}), the sequence 
$$
\big(\L^{(\cfreq)}   v_{\alpha,n}^{(\cfreq)}+\alpha  v_{\alpha,n}^{(\cfreq)}\big)(x), \quad n\in \Bbb N,
$$
has a normally converging subsequence. Since the limit of the sequence is known to exist for any $s\in \Bbb R_+$ (in this case we get a real-weighted graph), and the holomorphic function is uniquely determined by its values on the real line, we get that there exist the limit 
\begin{equation}\label{eq::firstLim}
\lim_{n\to \infty}\big(\L^{(\cfreq)}   v_{\alpha,n}^{(\cfreq)}+\alpha  v_{\alpha,n}^{(\cfreq)}\big)(x),
\end{equation} 
for any $\alpha\ge 0$, and this limit is a holomorphic function on $\Bbb H_r$.  
Further, since \eqref{eq::unifombound} implies uniform boundness of $\big(\L^{(\cfreq)}   v_{\alpha,n}^{(\cfreq)}+\alpha  v_{\alpha,n}^{(\cfreq)}\big)(x)$ for any $\alpha<1$, we obtain by the same line of arguments, that there exist a holomorphic limit
\begin{equation}\label{eq::secondLim}
\lim_{\substack{\alpha\to 0,\\\alpha>0}}\lim_{n\to\infty}\big(\L^{(\cfreq)}   v_{\alpha,n}^{(\cfreq)}+\alpha  v_{\alpha,n}^{(\cfreq)}\big)(x). 
\end{equation} 
Since the limit \eqref{eq::firstLim} for $\alpha=0$ and the limit \eqref{eq::secondLim} coincide on the real positive half-line (\cite[Theorem 6.26]{KellerLenzWojcBook}) we have:
\begin{equation}\label{eq::eqLim}\lim_{\substack{\alpha\to 0,\\\alpha>0}}\lim_{n\to\infty}\big(\L^{(\cfreq)}   v_{\alpha,n}^{(\cfreq)}+\alpha  v_{\alpha,n}^{(\cfreq)}\big)(x)=\lim_{n\to\infty}\big(\L^{(\cfreq)}   v_{0,n}^{(\cfreq)}\big)(x).
\end{equation} 
Moreover, since by Green's formula
$$
\Re \big(\L^{(\cfreq)}   v_{0,n}^{(\cfreq)}\big)(x)=\dfrac{\Re \Q(v^{(\cfreq)}_{0,n})}{m(x)}>0,
$$
for all $n>N$, for any $s\in\Bbb H_r$, by Hurwitz theorem (see, e.g. \cite[p. 152, 2.6 Corollary]{Conway}) the limit in \eqref{eq::eqLim} either has no zeros on the right half-plane or is identically zero.  Note that due to Theorem \ref{thm::reccap}, the latest case is equivalent to the recurrence of the graph $(X,b,m)$, where $b=a^{(s)}$ for (some) all $s\in \Bbb H_r$.

Now we  apply a similar argument to show the convergence of $v_{\alpha,n}^{(\cfreq)}(y)$, where $y\in X$ is fixed, i.e. it is enough to prove the uniform boundness of  $v_{\alpha,n}^{(\cfreq)}(y)$ for compact subsets of $\Bbb H_r$. Let $x=x_0\sim x_1\sim \dots \sim x_k=y$ be a path between $x$ and $y$. By Cauchy–Schwarz inequality,  \eqref{eq::Lv1cfreq}  and \eqref{eq::sumas} we get
\begin{align}\label{eq::estv}
\big|&v_{\alpha,n}^{(\cfreq)}(y)-1\big|^2=\big|v_{\alpha,n}^{(\cfreq)}(y)-v_{\alpha,n}^{(\cfreq)}(x)\big|^2\\
&\le
\left(
\sum_{i=1}^{k}
\big|v_{\alpha,n}^{(\cfreq)}(x_{i})-v_{\alpha,n}^{(\cfreq)}(x_{i-1})\big|
\,\sqrt{\Re\a^{(s)}(x_{i-1},x_i)}
\cdot
\frac{1}{\sqrt{\Re\a^{(s)}(x_{i-1},x_i)}}
\right)^2\notag \\
&\le \left(\sum_{i=1}^{k}
\big|v_{\alpha,n}^{(\cfreq)}(x_{i})-v_{\alpha,n}^{(\cfreq)}(x_{i-1})\big|^2
\Re\a^{(s)}(x_{i-1},x_i)\right)
\left (\sum_{i=1}^{k}\frac{1}{\Re\a^{(s)}(x_{i-1},x_i)}\right)\notag\\
&\le
\Big|\Q\big(v^{(\cfreq)}_{\alpha,n}\big)\Big|\,\left (\sum_{i=1}^{k}\frac{1}{\Re\a^{(s)}(x_{i-1},x_i)}\right)\notag\\
&\le \left(1+\dfrac{|\Im\cfreq|^2}{(\Re{\cfreq})^2}\right)\left(\sum_{\substack{y\in X:\\y\sim x}}\big|\a^{(\cfreq)}(x,y)\big|+\alpha\cdot m(x)\right)\left (\sum_{i=1}^{k}\frac{1}{\Re\a^{(s)}(x_{i-1},x_i)}\right)\notag\\
&\le \dfrac{|s|^2}{(\Re s)^2} \left(\frac{1 + |\cfreq|^2}{\Re \cfreq}C_1+C_2\right)\left (\sum_{i=1}^{k}\frac{1}{\Re\a^{(s)}(x_{i-1},x_i)}\right)\notag,
\end{align}
where $C_1, C_2$ does not depend on $s$. Let us estimate the last term, starting from an estimate of $\Re\a^{(s)}(z,w), z,w\in X, s\in \Bbb H_r$. By Definition \ref{def::elnetwork} of an electrical network we have
$$
\Re\a^{(s)}(z,w)=\dfrac{s}{L_{zw}s^2+R_{zw}s+D_{zw}}.
$$ 
We estimate, using $|s|^2\ge |s|\Re s$ for any $s\in \Bbb C$ in the second line:
\begin{align*}
\Re \a^{(s)}&(z,w)=\Re\dfrac{s(L_{zw}\overline s^2+R_{zw}\overline s+D_{zw})}{|L_{zw}s^2+R_{zw}s+D_{zw}|^2}\ge \Re\dfrac{s(L_{zw}\overline s^2+R_{zw}\overline s+D_{zw})}{(L_{zw}|s|^2+R_{zw}|s|+D_{zw})^2}\\
&=\Re\dfrac{L_{zw}|s|^2 \overline  s+R_{zw} |s|^2+D_{zw}s}{(L_{zw}|s|^2+R_{zw}|s|+D_{zw})^2}\ge (\Re s) \dfrac{L_{zw}|s|^2 +R_{zw} |s|+D_{zw}}{(L_{zw}|s|^2+R_{zw}|s|+D_{zw})^2}\\
&=(\Re s) \dfrac{1}{L_{zw}|s|^2+R_{zw}|s|+D_{zw}}\ge \dfrac{\Re s}{\max \{1,|s|,|s|^2\}} \dfrac{1}{L_{zw}+R_{zw}+D_{zw}}\\
&=\dfrac{\Re s}{\max\{ 1,|s|^2\}} \dfrac{1}{L_{zw}+R_{zw}+D_{zw}}\ge \dfrac{\Re s}{ 1+|s|^2} \dfrac{1}{L_{zw}+R_{zw}+D_{zw}}.
\end{align*}
since either $|s|<1$ or $|s|\le|s|^2$.
Combining this with \eqref {eq::estv} we obtain:
$$
\big|v_{\alpha,n}^{(\cfreq)}(y)-1\big|^2\le \dfrac{|s|^2(1+|s|^2)}{(\Re s)^3} \left(\frac{1 + |\cfreq|^2}{\Re \cfreq}C_1+C_2\right)C_3,
$$
where $C_3= \sum_{i=1}^{k}(L_{x_{i-1}x_i}+R_{x_{i-1}x_i}+D_{x_{i-1}x_i})$, $C_1,C_2,C_3$ do not depend on $s$, $n$ and any $\alpha<1$. Hence,  $v_{\alpha,n}^{(\cfreq)}(y)$ is uniformly bounded for any $n\in \Bbb N$  in any domain
$$
\{\Re \cfreq\ge\varepsilon, |\cfreq|\le C\}, \quad\varepsilon,  C\in \Bbb R^+.
$$

Therefore, again by Montel's theorem and existence of the limit on the real half-line, we conclude 
the existence and the equality of the limits
\begin{equation}\label{eq::eqLimv} 
\lim_{\substack{\alpha\to 0,\\\alpha>0}}\lim_{n\to\infty}v_{\alpha,n}^{(\cfreq)}(y)=\lim_{n\to\infty} v_{0,n}^{(\cfreq)}(y).
\end{equation}
Moreover, note that in the case $\Big|\Q\big(v^{(s)}_{0,n}\big)\Big|\to 0$ (i.e in the case of recurrence), we have $\lim_{n\to\infty} v_{0,n}^{(\cfreq)}(y)=1$ by the fourth line of \eqref{eq::estv}.

Finally,  \eqref{eq::LhateqL}, \eqref{eq::eqLim} and \eqref{eq::eqLimv} imply
$$
\lim_{\substack{\alpha\to 0,\\\alpha>0}}\lim_{n\to\infty}({\widehat L_n}+\alpha)^{-1}1_x(y)= \lim_{n\to\infty}({\widehat L_n})^{-1}1_x(y)=\lim_{n\to\infty}\dfrac{v_{0,n}^{(\cfreq_0)}(y)}{\big(\L^{(\cfreq_0)}   v_{0,n}^{(\cfreq_0)}\big)(x)},
$$
and the last limit is finite in the case of transient graph  and is $1/0=\infty$ in the case of recurrent graph. Due to \eqref{eq::GxyeqL} and \eqref{eq::LnN} the theorem is proven.
\end{proof}

The proof above immediately implies the following characterization of recurrence, which is known for the real-weighted graphs, where it can be proven by monotone convergence argument (see \cite[Chapter 6.4]{KellerLenzWojcBook}).
\begin{corollary}\label{cor::Greenrec}
Let $(X,b,m)$ be a complex-weighted graph. Then for its Green's function holds
$$
G(x,y)=\lim_{\substack{\alpha\to 0,\\\alpha>0}}(L^{(D)}+\alpha)^{-1}1_x(y)=\lim_{\substack{\alpha\to 0,\\\alpha>0}}\lim_{n\to\infty}({L_n}+\alpha)^{-1}1_x(y).
$$
Moreover, the graph is recurrent if and only if $G(x,y)=\infty$ for some (all) $x,y\in X$.
\end{corollary}

\subsection{Neumann Laplacian}
The aim of this section is to  introduce the Neumann Laplacian on infinite complex-weighted graphs, and show that it coincides with the Dirichlet Laplacian for all measures if and only if the graph is recurrent. This result is known for real-weighted graphs and we use the same line of arguments, see \cite[Theorem 6.1]{KellerLenzWojcBook}. 

Let $(X,b,m)$ be a complex-weighted graph. We define the form $Q^{(N)}$ as a restriction of the formal form $\Q$ to $\mathcal D\cap \ell^2(X,m)$. We will call $Q^{(N)}$ the {\em form corresponding to the Neumann Laplacian}.

\begin{lemma}
The form $Q^{(N)}$ with the domain $D(Q^{(N)})=\mathcal D\cap \ell^2(X,m)$ is a densely defined closed sectorial form in $\ell^2(X,m)$.
\end{lemma}
\begin{proof}
The sectoriality follows from Lemma \ref{lem::formalQmapstosector}. The density is clear, since $C_c(X)\subset \D$ is dense in $\ell^2(X,m)$. To  prove the closedness, we use the fact that  the lower semi-continuity implies closedness (see \cite[Theorem B.9, proof (ii)$\Rightarrow$(iii)]{KellerLenzWojcBook}). Since $\Re \Q(f)=\Q_{\Re b}(f)$, using lower-semicontinuity of the form $\Q_{\Re b}$ on the real-weighted graph $(X,\Re,b,m)$ (see \cite[Proposition 1.3]{KellerLenzWojcBook} for the details) and  Lemma~\ref{lem::complexf}(2),  we obtain that the form $\Q=\Q_b$ is lower semi-continuos, i.e.
\begin{equation}\label{eq:fQlow}
\Re \Q(f)\le \lim \inf_{n\to \infty} \Re \Q(f_n)
\end{equation}
for any $f, f_n\in \D, n\in \Bbb N$ such that $f_n(x)\to f(x), n\to \infty$ pointwise for all $x\in X$. Since $Q^{(N)}$ is a restriction of $\Q$ and convergence in $\ell^2(X,m)$ implies pointwise convergence, we get from \eqref{eq:fQlow} that 
$$
\Re Q^{(N)}(f)\le \lim \inf_{n\to \infty} \Re Q^{(N)}(f_n)
$$
for any  $f\in \ell^2(X,m)$ and $(f_n)_{n\in \Bbb N}$ converging to $f$ in $\ell^2(X,m)$. Therefore, $Q^{(N)}$ with the domain $D(Q^{(N)})=\mathcal D\cap \ell^2(X,m)$ is also lower semi-continuous, and, hence, closed.
\end{proof}

By the general theory of sectorial forms and operators, see e.g. \cite{Kato}, the form $Q^{(N)}$ defines an $m$-sectorial operator, which we call  the {\em Neumann Laplacian} and denote by $L^{(N)}$, with a domain $D(L^{(N)})\subset D(Q^{(N)})$, see \cite[p. 322, Theorem 2.1]{Kato}.
Following the same outline as in the proof of Lemma~\ref{lem::LK=mathcalL} we can show, that $L^{(N)}f(x)=\mathcal Lf(x)$ for all $f\in D(L^{(N)})$ and any $x\in X$.

\begin{lemma}\label{lem::DlDell2}
Let $(X,b,m)$ be a graph. Then  $D\big(Q^{(D)}\big)=D\big(Q^{(N)}\big)$ if and only if
$$
D(L^{(D)})=\{f\in D(Q^{(N)})\;\mid\; \L f\in \ell^2(X,m)\}.
$$

\end{lemma}

\begin{proof}
By the definition of the associated operator, $L^{(D)}$ maps to $\ell^2(X,m)$, Hence, since its action coincides with the action of the formal Lapalcian, the inclusion
\begin{equation*}\label{eq::incl}
D(L^{(D)})\subset\{f\in D(Q^{(D)})\;\mid\; \L f\in \ell^2(X,m)\},
\end{equation*}
always holds. This immediately implies the inclusion 
$$
D(L^{(D)})\subset\{f\in D(Q^{(N)})\;\mid\; \L f\in \ell^2(X,m)\},
$$
since $D(Q^{(D)})\subset D(Q^{(N)})$ by the definitions of the domains.
\begin{itemize}[leftmargin=*]
\item
Let $D\big(Q^{(D)}\big)=D\big(Q^{(N)}\big)$.  By the above we need to prove 
$$
\{f\in D(Q^{(N)})\;\mid\; \L f\in \ell^2(X,m)\}\subset D(L^{(D)}).
$$
Let $f\in D(Q^{(N)})=D(Q^{(D)})$. Then by Greens' formula we have
$$
Q^{(D)}(f, \phi)=(\L f\;\mid\;\phi),
$$
for all $\phi\in C_c(X)$. As $D(Q^{(D)})=\overline{C_c(X)}^{\|\cdot\|_\mathcal Q}$  by Theorem \ref{thm::dol2cQ}, we conclude
$$
Q^{(D)}(f, g)=(\L f\;\mid\;g),
$$
for all $g\in D(Q^{(D)})$. Hence, $f\in D(L^{(D)})$.
\item
Let $D(L^{(D)})=\{f\in D(Q^{(N)})\;\mid\; \L f\in \ell^2(X,m)\}.$ By the definition of the associated operator, $L^{(N)}$ maps to $\ell^2(X,m)$. Hence, since its action coincides with the action of the formal Lapalcian we conclude
$$
D(L^{(N)})\subset \{f\in D(Q^{(N)})\;\mid\; \L f\in \ell^2(X,m)\}=D(L^{(D)}).
$$
As both operators are restrictions of $\L$ and $m$-sectorial, we conclude $D(L^{(N)})= D(L^{(D)}).$ Hence, the associated forms also coincide.
\end{itemize}
\end{proof}

\begin{theorem}
A complex-weighted graph $(X,b,m)$ is recurrent if and only if any of the following equivalent conditions hold.
\begin{itemize}
\item[{\em(i)}]
$\D=\D_0$.
\item[{\em(ii)}]
for all measures $\mu:X\to \Bbb R^+$ the domains $D\big(Q_\mu^{(D)}\big)$ and $D\big(Q_\mu^{(N)}\big)$ of the forms on graph $(X,b,\mu)$,  coincide.
\item[{\em(iii)}]
for all measures $\mu:X\to \Bbb R^+$:
$$
D(L_\mu^{(D)})=\{f\in D(Q_\mu^{(N)})\;\mid\; \L f\in \ell^2(X,\mu)\},
$$
where $L_\mu^{(D)}$ is the Dirichlet Laplacian and $Q_\mu^{(N)}$ is a form, corresponding to the Neumann Laplacian on the graph $(X,b,\mu)$.
\end{itemize}

\end{theorem}

\begin{proof}

\textit{Recurrence}$\Rightarrow$(i). Let $(X,b,m)$ be recurrent. By Corollary \ref{cor::Rebrec} the graph $(X,\Re b,m)$ is also recurrent. Hence, any real-valued function from $\D_0$ is in $\D$ (see \cite[Theorem 6.1 (i.a)]{KellerLenzWojcBook}). Now (i) follows from Lemma \ref{lem::fRefImfH} and Lemma \ref{lem::complexf}. 

(i)$\Rightarrow$(ii) since $D\big(Q_\mu^{(D)}\big)=\D_0\cap\ell^2(X,\mu)$ and $D\big(Q_\mu^{(D)}\big)=\D\cap\ell^2(X,\mu)$ by the definitions of the corresponding forms.

(ii)$\Rightarrow$\textit{Recurrence}. Taking a finite measure $\mu$ (i.e. $\mu(X)<\infty$), we get $1\in \ell^2(X,\mu)$. Hence, $1\in D(Q_\mu^{(N)})=D(Q_\mu^{(D)})$, from where follows $1\in \D_0$ and recurrence by Theorem \ref{thm::1inD0}.

Finally, (ii)$\Leftrightarrow $(iii) due to Lemma \ref{lem::DlDell2}.

\end{proof}

\section*{Acknowledgments}
The author thanks Philipp Bartmann  and Matthias Keller for fruitful discussions and helpful  comments on the topic.

\end{document}